\documentclass[11pt]{amsart}
  \usepackage{graphicx,amsfonts,layout}

\usepackage{amssymb}
\usepackage{amscd} 
\usepackage{amsmath}
\usepackage[all]{xy}
\usepackage{mathrsfs} 
\usepackage{graphicx}
 \usepackage{color}

\theoremstyle{plain}
\newtheorem{theorem}{Theorem}[section]
\newtheorem{lemma}[theorem]{Lemma}

\newtheorem{proposition}[theorem]{Proposition}
\newtheorem{definition}[theorem]{Definition}

\newtheorem{example}[theorem]{Example}

\newtheorem{remark}[theorem]{Remark}

  1

\DeclareMathOperator{\Gal}{Gal}

\DeclareMathOperator{\N}{N}

\DeclareMathOperator{\im}{im}

\newcommand{\CC}{\mathbb{C}}

\newcommand{\QQ}{\mathbb{Q}}

\newcommand{\ZZ}{\mathbb{Z}}

\newcommand{\G}{\mathcal{G}}

\def\bigcapp{\raise1ex\hbox{\rotatebox{180}{$\biguplus$}}}
 \def\bigcappd{\raise1ex\hbox{\rotatebox{180}{$\displaystyle\biguplus$}}}

\begin{document}

\title[]{On a conjecture of Coleman\\
concerning Euler systems}

 \author{David Burns, Alexandre Daoud and Soogil Seo}

\begin{abstract} We prove a distribution-theoretic conjecture of Robert Coleman, thereby also obtaining an explicit description of the complete set of Euler systems for the multiplicative group over 
$\QQ$. 
\end{abstract}

\address{King's College London,
Department of Mathematics,
London WC2R 2LS,
U.K.}
\email{david.burns@kcl.ac.uk}

\address{King's College London,
Department of Mathematics,
London WC2R 2LS,
U.K.}
\email{alexandre.daoud@kcl.ac.uk}

\address{Yonsei University,
Department of Mathematics,
Seoul,
Korea.}
\email{sgseo@yonsei.ac.kr}


\maketitle

\section{Introduction and statement of main result}\label{intro}

The theory of distributions plays a prominent role in number theory research and is strongly influenced by the classical theory of circular numbers in abelian fields (see, for example, the discussion of Kubert and Lang in the Introduction to \cite{kl}). In this article, we shall prove a distribution-theoretic conjecture of Robert Coleman that has a natural 
interpretation in terms of the existence of Euler systems for the multiplicative group $\mathbb{G}_m$.

To be more precise, we write $\QQ^c$ for the algebraic closure of $\QQ$ in $\CC$ and $\mu^\ast$ for the set of non-trivial roots of unity in $\QQ^c$. We then recall that a `circular distribution' is a $\Gal(\QQ^c/\QQ)$-equivariant function $f: \mu^* \to \QQ^{c,\times}$ with the property that 
\begin{equation}\label{dist1} \prod_{\zeta^a = \varepsilon} f(\zeta) = f(\varepsilon)\end{equation}
for all natural numbers $a$ and all elements $\varepsilon$ of $\mu^*$. 

In the late 1980's, Coleman formulated a remarkably explicit conjectural description of the complete set of circular distributions. This conjectural description (which we henceforth refer to as `Coleman's Conjecture') was directly motivated by an archimedean characterization of circular units obtained in \cite{coleman2} and was therefore related to attempts to understand a globalized version of the fact that all norm-compatible families of units in towers of local cyclotomic fields arise by evaluating a Coleman power series at roots of unity, as had earlier been proved by Coleman in \cite{coleman}.

To verify Coleman's Conjecture, we find it convenient to reinterpret the conjecture in terms of a suitable notion of Euler systems. For this purpose we write $\Omega$ for the set of finite abelian extensions of $\QQ$ in $\QQ^c$. For $E$ and $E'$ in $\Omega$ with $E \subseteq E'$ we write ${\N}_{E'/E}$ for the field-theoretic norm map $(E')^\times \to E^\times$. For a Galois extension $E$ of $\QQ$ in $\QQ^c$ we set $\G_E := \Gal(E/\QQ)$. For a rational prime $\ell$ we write $\sigma_\ell$ for the {\em inverse} Frobenius automorphism of $\ell$ on the maximal subextension of $\QQ^c$ in which $\ell$ is unramified. 
 Using this notation, we can now define the Euler systems that are relevant to our purposes.   
 
\begin{definition}\label{es def}
An Euler system for $\mathbb{G}_m$ over $\QQ$ is a collection
$$u=(u_E)_{E  } \in \prod_{E \in \Omega} E^\times $$
with the property that for every $E$ and $E'$ in $\Omega$ with $E \subset E'$ one has
\begin{equation}\label{classical dr}{\N}_{E'/E}(u_{E'})= (u_{E})^{\prod_{\ell} (1-\sigma_\ell)} \end{equation}
in $E^\times$, where in the product $\ell$ runs over the finite set of prime numbers that ramify in $E'$ but not in $E$. We write ${\rm ES}(\mathbb{G}_m)$ for the collection of all such systems.
\end{definition}

\begin{remark}\label{structure rem} {\em The set ${\rm ES}(\mathbb{G}_m)$ is an abelian group under multiplication of systems (so that the component of $u_1u_2$ at $E$ is equal to the product $u_{1,E}u_{2,E}$), with identity element equal to 
the system $u_{\rm triv}$ that has $u_{{\rm triv},E} = 1$ for every $E$ in $\Omega$. This group is also endowed with a natural action of the commutative, non-noetherian, ring 
\[ R := \varprojlim_{E\in \Omega}\ZZ[\G_E],\]
where the transition morphisms in the inverse limit are induced by the natural restriction maps $\ZZ[\G_{E'}] \to \ZZ[\G_E]$ for $E \subset E'$. For $u = (u_E)_E$ in ${\rm ES}(\mathbb{G}_m)$ and $r = (r_E)_E$ in $R$ we write $u^r$ for the system in ${\rm ES}(\mathbb{G}_m)$ that has value $u_E^{r_E}$ at each $E$ in $\Omega$. In a similar way, the set of circular 
distributions also has a natural multiplicative structure as $R$-module.}\end{remark}

%
%
%

\begin{remark}{\em We caution the reader that, whilst we usually use exponential notation to indicate the action of a commutative ring $\Lambda$ on a multiplicative group $U$, as in (\ref{classical dr}) and Remark \ref{structure rem},  we shall also often, for typographic simplicity, use additive notation and hence write either $\lambda(u)$ or $\lambda\cdot u$ in place of $u^\lambda$ for  $\lambda\in\Lambda$ and $u\in U$.}
 \end{remark}

To describe some explicit Euler systems (in the above sense),  we set 
\[ \zeta_n := e^{2\pi i/n} \,\,\text{ and }\,\, \QQ(n) := \QQ(\zeta_n)\subset \QQ^c\]
for each natural number $n$. We also write $m(E)$ for the finite part of the conductor of each field $E$ in $\Omega$ (so that $E \subseteq \QQ(m(E))$).

\begin{example}\label{ex1}{\em For $E$ in $\Omega$ set 
$$c_E :={\N}_{\QQ(m(E))/E}(1-\zeta_{m(E)}) \in E^\times.$$
Then by substituting $X = \zeta_{mn}$ in the polynomial identities 
$1 - X^n = \prod_{a=0}^{a=n-1} (1- \zeta_{mn}^{ma}X)$ for each pair of natural numbers $m$ and $n$ one checks that the `cyclotomic Euler system' 
\[ c := (c_E)_{E\in \Omega}\]
satisfies the distribution relations (\ref{classical dr}) and so    belongs to ${\rm ES}(\mathbb{G}_m)$. This system is known to be directly related to the values of derivatives of Dirichlet $L$-series (cf. \cite[Chap. 3, \S5]{tate}). } \end{example}

\begin{example}\label{ex11}{\em Let $\mathcal{P}$ denote the collection of non-empty subsets of the set of all odd prime numbers. For $\Pi$ in $\mathcal{P}$ and $E$ in $\Omega$ set 
$$u_{\Pi,E} :={\N}_{\QQ(m(E))/E}(-1)^{n_{\Pi,E}} \in \{\pm 1\},$$
with $n_{\Pi,E}$  defined to be $1$ if $m(E)$ is divisible only by primes in $\Pi$ and to be $0$ otherwise. Then an explicit check shows that, for each such set $\Pi$, the `Coleman distribution' 
\[ u_\Pi := (u_{\Pi,E})_{E\in \Omega}\]
satisfies the distribution relations (\ref{classical dr}) and hence defines an element of ${\rm ES}(\mathbb{G}_m)$ of order two. In the sequel we consider the $R$-submodule 
\[ T := R\cdot \{u_\Pi: \Pi \in \mathcal{P}\}\]
of ${\rm ES}(\mathbb{G}_m)$ that comprises all finite products of systems $u_\Pi$ for $\Pi$ in $\mathcal{P}$.}
\end{example}


Having recalled these concrete examples, we can now state an explicit description of the full module ${\rm ES}(\mathbb{G}_m)$.

\begin{theorem}\label{main result} One has ${\rm ES}(\mathbb{G}_m) =  T \oplus Rc$. \end{theorem}

This is our main result and, as far as we are aware, is the first explicit description of the complete set of Euler systems in any setting. Further, since Theorem \ref{main result} implies, modulo minor $2$-torsion issues, that {\em every} Euler system for $\mathbb{G}_m$ over $\QQ$ is directly related to the cyclotomic Euler system, and hence to Dirichlet $L$-series, it  demonstrates the remarkable strength of distribution relations and thereby perhaps helps to explain the great difficulty that there is to construct Euler systems in natural settings. 

For each odd prime $p$, the result of Theorem \ref{main result} also implies that the pro-$p$ completion of ${\rm ES}(\mathbb{G}_m)$ is generated over the pro-$p$ completion $R_p$ of $R$ by (the image of) the system $c$, and hence validates an analogue for $\mathbb{G}_m$ of the question of whether the $R_p$-module of $p$-adic Euler systems for $\ZZ_p(1)$ is cyclic, as asked by Mazur and Rubin at the end of \cite[\S5.3]{MRkoly}. 

Turning to Coleman's Conjecture, we note that for any circular distribution $f$ there exists a unique Euler system $u_f$ in ${\rm ES}(\mathbb{G}_m)$ with the property that  $u_{f,\QQ(m)} = f(\zeta_m)$ for all $m > 1$ with $m \not\equiv 2$ (mod $4$). The assignment $f \mapsto u_f$ constitutes an isomorphism between the $R$-module of circular distributions and ${\rm ES}(\mathbb{G}_m)$ (for details see the discussion in \cite[\S2.1.1]{yonsei}) and, via this isomorphism, the result of Theorem \ref{main result} can be seen to validate the precise statement of Coleman's Conjecture (as recalled explicitly, for example, in \cite[Conj. 1.1]{yonsei}).    

For the same reason, the description in Theorem \ref{main result} also implies an affirmative answer to the `Guess' formulated by the third author in \cite[\S3]{Seo4}, thereby providing a natural distribution-theoretic analogue of the main result of Coleman in \cite{coleman2}. 

In addition, if $K$ is the maximal real subfield of $\QQ(n)$ for any natural number $n$, then the discussion of the third author in \cite[\S1]{Seo5} shows that Theorem \ref{main result} combines with results of Sinnott \cite{sinnott} on cyclotomic units  to imply that the cardinality of the graded module of `truncated Euler systems' over $K$ that is defined in \cite{Seo5} is equal to the class number of $K$, as is conjectured in loc. cit. 

Finally we note that, in a complementary article, it will be shown that Theorem \ref{main result} gives concrete information about the structure over (the non-noetherian ring) $R$ of the Selmer group of $\mathbb{G}_m$ over the abelian closure of $\QQ$. In this regard, we recall that the latter Selmer group is a classical object in number theory that encodes information about the Galois structures of the ideal class group and unit group of every abelian field. 

\begin{remark}{\em Following Coleman, a circular distribution $f$ is said to be `strict' if for all natural numbers $n$ and all primes $\ell$ that do not divide $n$ it satisfies the congruence relation 
\[ f(\zeta_\ell\zeta_n) \equiv f(\zeta_n) \,\,\text{ modulo all primes above $\ell$.}\]
Such congruence relations also arise naturally in the theory of Euler systems (see, for example, the discussion in \cite{RL}). The collection of strict circular distributions corresponds (under the isomorphism discussed above) to the $R$-submodule ${\rm ES}(\mathbb{G}_m)^{\rm str}$ of ${\rm ES}(\mathbb{G}_m)$ comprising systems $u = (u_E)_{E\in \Omega}$ with the property that for all fields $E \subset E'$ for which $m(E') = \ell\cdot m(E)$ with $\ell$ a prime that does not divide $m(E)$ one has $u_{E'} \equiv u^{\sigma_\ell}_{E} \,$ modulo all primes above $\ell$. With $u_{\rm odd}$ denoting the Coleman distribution corresponding (via the discussion in Example \ref{ex11}) to the set $\Pi_{\rm odd}$ of all odd primes, it follows as an easy consequence of Theorem \ref{main result} that 
\[ {\rm ES}(\mathbb{G}_m)^{\rm str} = \{u_{\rm triv}, u_{\rm odd}\} \oplus Rc.\] } \end{remark} 

\begin{remark}{\em In \cite{coates} Coates introduced an analogue of the notion of circular distributions in the setting of abelian extensions of imaginary quadratic fields and it seems likely the methods used here could be further developed in order to prove an analogue of Theorem \ref{main result} in that setting.}\end{remark}  


\section{The proof of Theorem \ref{main result}}\label{es q sec}

Our proof of Theorem \ref{main result} will follow a general approach to Coleman's Conjecture that was developed by the first and third authors in \cite{yonsei}.  

In contrast to this earlier approach, however, we focus here on the study of individual Euler systems rather than on families of Euler systems and are thereby able to reduce verification of the conjecture to a natural $p$-adic problem for every prime $p$. This key reduction is explained in \S\ref{pro p reduction} and relies critically both on aspects of the Galois structure of modules of cyclotomic elements (that follow from the link between cyclotomic elements and Dirichlet $L$-series) and the fact that the Euler systems that are relevant to Coleman's Conjecture have components in abelian fields of arbitrary conductor. 

Having made this reduction, the individual $p$-adic problems are then resolved by combining a 
delicate analysis of pro-$p$ completions together with results from \cite{yonsei} which themselves rely on detailed properties of Euler systems that are established by Rubin in \cite{R} and by Greither in \cite{greither}. 

In the sequel, we write $A_{\rm tor}$ for the torsion subgroup of an abelian group $A$.

\subsection{Initial observations}\label{initial section} At the outset we recall it is proved by the third author in \cite[Th. 2.5]{Seo4} that the group $Rc$ is torsion-free and hence that $T\cap Rc = \{0\}$. 

To prove Theorem \ref{main result} it is therefore enough for us to show that each system $v$ in ${\rm ES}(\mathbb{G}_m)$ belongs to $T + Rc$. Our proof of this fact (for a system $v$ that is henceforth regarded as fixed) will occupy the remainder of this article. 

In this first section we make several useful deductions from results of \cite{yonsei}. To do this we write $\tau$ for the element of $\G_{\QQ^c}$ induced by complex conjugation and then define $R$-modules   
\[ C := Rc^{1+\tau},\,\,Y:= Rv^{1+\tau} \,\,\text{ and }\,\, X:= (C+Y)/C.\]

The following result shows that Theorem \ref{main result} is equivalent to asserting that the module $X$ vanishes.

\begin{lemma}\label{injectivity}  There exists a canonical exact sequence  
\begin{equation*} 0 \to T + Rc \xrightarrow{\subset} T + Rc + Rv \xrightarrow{t} X\to 0\end{equation*}
\end{lemma}

\begin{proof} Since $y^{1+\tau} = 0$ for every $y\in T$ one has $z^{1+\tau}\in C+Y$ for each $z \in T + Rc + Rv$ and so one obtains a well-defined surjective homomorphism of $R$-modules $t$ by sending each $z$ to the class of $z^{1+\tau}$ in $X$. 

With this definition of $t$, it is clear $T+Rc$ is contained in $\ker(t)$ and hence enough to show that if $t(z)=0$, then $z$ belongs to $T + Rc$. 
 
Now if $t(z) = 0$, then there exists an element $r$ of $R$ such that $z^{1+\tau} = (c^{1+\tau})^r$. It follows that $(zc^{-r})^{1+\tau} = 1$ and hence, by \cite[Th. 4.1(i)]{yonsei}, that $zc^{-r}$ belongs to ${\rm ES}(\mathbb{G}_m)_{\rm tor} + Rc^{1-\tau}$. Since this implies $z$ belongs to ${\rm ES}(\mathbb{G}_m)_{\rm tor} + Rc$, it is therefore enough to note that in \cite[Th. B]{Seo3} the third author has proved ${\rm ES}(\mathbb{G}_m)_{\rm tor}$ is equal to $T$. \end{proof}  

The following property of $X$ will also play a key role in the sequel. 

\begin{lemma}\label{tf lemma} $X$ is torsion-free. \end{lemma}

\begin{proof} The module $X$ identifies with a subgroup of the quotient $Q$ of ${\rm ES}(\mathbb{G}_m)^{1+\tau}$ by $Rc^{1+\tau}$. In addition, one knows that $Q$ is torsion-free since, for every prime $p$, it is isomorphic to a subgroup of a uniquely $p$-divisible group, as a direct consequence of claims (i) and (ii) of \cite[Th. 5.1]{yonsei}. \end{proof} 

\subsection{Annihilators of cyclotomic units}\label{ann cyclo section} In this section we prove some useful technical results concerning the Galois structure of modules generated by Euler systems. 

We write $\#X$ for the cardinality of a finite set $X$. If $\Gamma$ is a finite group, then we write $e_\Gamma$ for the idempotent $e_\Gamma := \#\Gamma^{-1}\cdot \sum_{\gamma \in \Gamma}\gamma$ of $\QQ[\Gamma]$, and for each homomorphism $\chi: \Gamma \to \QQ^{c,\times}$ we write $e_\chi$ for the primitive idempotent $(\#\Gamma)^{-1}\sum_{\gamma\in \Gamma}\chi(\gamma^{-1})\gamma$ of $\QQ^c[\Gamma]$.

For each field $L$ in $\Omega$ we write $L^+$ for its maximal real subfield and set $\G_L^+ := \G_{L^+}$. We then define an ideal of $\ZZ[\G_L^+]$ by setting 
\[ I_L := \{ r \in \ZZ[\G_L^+]: r(c^{1+\tau}_L) = 0\}.\]

In the following result we shall describe explicitly this annihilator ideal in terms of the idempotent of $\QQ[\G_L^+]$ that is obtained by setting
\begin{equation}\label{n idem} e_L := \begin{cases} 1, &\text{if $m(L)$ is a prime power,}\\
                        \prod_{\ell \mid m(L)}(1-e_{D_{L,\ell}}), &\text{otherwise,}\end{cases}\end{equation}
where in the product $\ell$ runs over all prime divisors of $m(L)$ and $D_{L,\ell}$ denotes the decomposition subgroup of $\ell$ in $\G^+_L$.


\begin{proposition}\label{useful 3} For every field $L$ in $\Omega$ the following claims are valid. 

\begin{itemize}
\item[(i)] $I_L$ is equal to the set $\{ x\in \ZZ[\G^+_L] \, \mid \, e_L\cdot x =0\}$.
\item[(ii)] If $\psi: \G^+_L \to \QQ^{c,\times}$ is any homomorphism such that $e_\psi e_L = 0$, then $m(L)$ is not a prime power and $\psi$ is trivial on the decomposition group in $\G_L^+$ of at least one prime divisor of $m(L)$.
\item[(iii)] If $u$ belongs to ${\rm ES}(\mathbb{G}_m)^{1+\tau}$, then the image of $u_L$ in $\QQ\otimes_\ZZ L^\times$ belongs to the $\QQ[\G^+_L]$-module generated by $c^{1+\tau}_L$. 
\end{itemize}
\end{proposition}

\begin{proof} Claim (i) is proved in \cite[Lem. 2.4]{yonsei} and relies on the fundamental link  between cyclotomic elements and first derivatives of Dirichlet $L$-series (as discussed, for example, in \cite[Chap. 3, \S5]{tate}). 

Claim (ii) follows directly from the explicit description (\ref{n idem}) of $e_L$ and the fact that for each subgroup $H$ of $\G^+_L$ one has $e_\psi(1-e_H)   = 0$ if $\psi$ is trivial on $H$ and $e_\psi(1-e_H) = e_\psi$ otherwise.
  
To prove claim (iii) we use the fact that the natural map $\iota: L^\times \to \QQ^c\otimes_\ZZ L^\times$ is injective on the torsion-free subgroup $(L^\times)^{1+\tau}$ of $L^{\times}$. We write $u = w^{1+\tau}$ with $w \in {\rm ES}(\mathbb{G}_m)$ and claim first that the image of $u_L = w_L^{1+\tau}$ under $\iota$ is stable under multiplication by $e_L$.
 In view of claim (ii), to show this it is enough to prove for every homomorphism $\psi: \G_L^+ \to \QQ^{c,\times}$ that if $e_\psi(\iota(u_L))\not= 0$, then $\psi$ cannot be trivial on the decomposition group of any prime that ramifies in $L$ (and so $e_\psi e_L = e_\psi$). 
 
 To see this, we write $\pi$ for the restriction map $\G_L \to \G_L^+$ and then note that, for each homomorphism $\psi: \G_L^+\to \QQ^{c,\times}$, one has 
\begin{align*} e_\psi(\iota(u_L)) =&\, e_{\psi\circ\pi}(\iota(w_L)^{1+\tau})\\
 =&\, 2\cdot e_{\psi\circ\pi}(\iota(w_L))
 \\ =&\, 2\cdot\left(\prod_{\ell\in \mathcal{P}_\psi} (1-\psi(\sigma_{\ell,L_\psi}))\right)e_{\psi\circ\pi}(\iota(w_{L_\psi}))\\
 =&\, \left(\prod_{\ell\in \mathcal{P}_\psi} (1-\psi(\sigma_{\ell,L_\psi}))\right)e_{\psi}(\iota(u_{L_\psi})).
 \end{align*}
Here $L_\psi$ denotes the subfield of $L$ fixed by $\ker(\psi\circ\pi)$ (or equivalently, the subfield of $L^+$ fixed by $\ker(\psi)$), $\mathcal{P}_\psi$ is the set of primes that ramify in $L$ but not in $L_\psi$ and for each $\ell$ in $\mathcal{P}_\psi$ we write $\sigma_{\ell,L_\psi}$ for the image of $\sigma_\ell$ in $\G_{L_\psi}$. In addition, the first of the equalities is clear, the second and fourth are true since the image of $\tau$ in $\G_L$ is contained in 
$\ker(\psi\circ\pi)$, and the third equality is true since the system $w$ validates the distribution relation (\ref{classical dr}).  

From the above equalities it is clear that, if $e_\psi(\iota(u_L))\not= 0$, then $\ker(\psi)$ cannot contain $\sigma_{\ell,L_\psi}$ for any $\ell$ in $\mathcal{P}_\psi$. On the other hand, any prime $\ell$ that ramifies in $L$ but does not belong to $\mathcal{P}_\psi$ is ramified in $L_\psi$ and so its inertia group in $\G_L^+$ is not contained in $\ker(\psi)$. Hence, if $e_\psi(\iota (u_L))\not= 0$, then $\psi$ cannot be trivial on the decomposition group in $\G_L^+$ of any prime that ramifies in $L$, as required.
 
To proceed we write $S(L)$ for the set of archimedean places of $L$, respectively the set of places of $L$ that are either archimedean or $p$-adic, if $m(L)$ is divisible by two distinct primes, respectively if $m(L)$ is a power of a prime $p$. We  then write $U_L'$ for the (finitely generated) subgroup of $L^\times$ comprising elements that are units at all places of $L$ outside $S(L)$ and $X_L'$ for the subgroup of the free abelian group on $S(L)$ comprising elements whose coefficients sum to zero. 

Then we recall that the distribution relation (\ref{classical dr}) implies $w_L$ belongs to $U_L'$ (for a proof of this fact see, for example, \cite[Lem. 2.2]{Seo1}) and hence, in view of the above argument, that $\iota(u_L)$ belongs to $e_L(\QQ\otimes_{\ZZ}(U_L')^{1+\tau})$. 
 
To prove claim (iii) it is thus enough to show that the $\QQ[\G^+_L]$-module $e_L (\QQ\otimes_{\ZZ}(U_L')^{1+\tau})$ is generated by $c^{1+\tau}_L$. But this is true since if $\psi$ is any homomorphism $\G_L^+ \to \QQ^{c,\times}$ for which $e_\psi e_L \not= 0$, then claim (i) combines with the observation that $c_L$ belongs to $U_L'$ to imply that $e_\psi(\iota(c^{1+\tau}_L))$ is a non-zero element of $e_\psi(\QQ^c\otimes_\ZZ (U_L')^{1+\tau})$, whilst one also knows that  
\[ {\rm dim}_{\QQ^c}\bigl(e_\psi(\QQ^c\otimes_\ZZ (U_L')^{1+\tau})) = {\rm dim}_{\QQ^c}\bigl(e_{\psi\circ\pi}(\QQ^c\otimes_\ZZ X_L')) = 1,\]
where the first equality is true since the Dirichlet Regulator map induces an isomorphism of $\CC[\G_L]$-modules $\CC\otimes_\ZZ U_L' \cong \CC\otimes_\ZZ X_L'$ (cf. \cite[Chap. I, \S4.2]{tate}) and the second follows by a straightforward computation from the definition of $X_L'$. 
\end{proof}

\subsection{The reduction of Theorem \ref{main result} to $p$-primary considerations}\label{pro p reduction} By the discussion in \S\ref{initial section}, the proof of Theorem \ref{main result} is reduced to showing that the group $X$ vanishes. In this section we reduce the vanishing of $X$ to a family of $p$-primary problems.  

\subsubsection{}For an abelian group $A$ we set 
\[ \widehat A := \varprojlim_{n\in \mathbb{N}} A/nA \,\,\text{ and }\,\, \widehat{A}^p := \varprojlim_{m\in \mathbb{N}} A/p^mA\] 
for each prime $p$, where all limits are taken with respect to the natural projection maps (and, for typographic simplicity, we sometimes write ${A}^{\wedge,p}$ in place of $\widehat{A}^p$). We use similar notation for homomorphisms of abelian groups. We also note that if $A$ is finitely generated, then $\widehat A$ and $\widehat{A}^p$ respectively identify with the tensor products $A\otimes_\ZZ {\widehat{\ZZ}}$ and $A\otimes_\ZZ \ZZ_p$.

The following result records some elementary properties of the functors $A \mapsto \widehat{A}$ and $A \mapsto \widehat{A}^p$ that will be useful in the sequel. 

\begin{lemma}\label{pro-completion lemma} If $A$ is a torsion-free abelian group, then the following claims are valid.
\begin{itemize}
\item[(i)] If $0 \to A_1 \xrightarrow{\theta} A_2 \xrightarrow{\phi} A \to 0$ is an exact sequence of abelian groups, then  the induced sequences $0 \to \widehat{A_1} \xrightarrow{\widehat{\theta}} \widehat{A_2} \xrightarrow{\widehat{\phi}} \widehat{A} \to 0$, and $0 \to \widehat{A_1}^p \xrightarrow{\widehat{\theta}^p} \widehat{A_2}^p \xrightarrow{\widehat{\phi}^p} \widehat{A}^p \to 0$ for each prime $p$, are also exact. 
\item[(ii)] The groups $\widehat{A}$, and $\widehat{A}^p$ for each prime $p$, are torsion-free. 
\item[(iii)] For each prime $p$, the natural map $\widehat{A}^p \to (\widehat{A}^p)^{\wedge,p}$ is bijective.
\end{itemize}
\end{lemma}

\begin{proof} For both claims (i) and (ii), it it enough to consider the functor $A \to \widehat{A}$. 

To prove claim (i) in this case we note first that, since $A$ is torsion-free, for each natural number $n$ the Snake Lemma applies to the following exact commutative diagram  

\[\begin{CD} 0 @> >> A_1 @>\theta >> A_2 @>\phi >> A @> >> 0\\
& & @V n VV @V n VV @V n VV\\
0 @> >> A_1 @>\theta >> A_2 @>\phi >> A @> >> 0\end{CD}\]
to give an exact sequence $0 \to A_1/nA_1 \xrightarrow{\theta/n} A_2/nA_2 \xrightarrow{\phi/n} A/nA \to 0$. It is then enough to note that the latter sequences are compatible (with respect to the natural projection maps) as $n$ varies and that, by the Mittag-Leffler criterion, exactness of the sequences is preserved when one passes to the inverse limit over $n$ since, for each multiple $m$ of $n$ the projection map $A_1/mA_1 \to A_1/nA_1$ is surjective. 

To prove claim (ii) we must show that if $x = (x_n)_n$ is an element of $\widehat{A}$ with the property that $px = 0$ for some prime $p$, then $x=0$. But, since $A$ is torsion-free, for each $n$ the element $x_{np}$ is the image in $A/(npA)$ of an element $\hat x_{np}$ of $nA$. Since $x_n$ is equal to the image of $\hat x_{np}$ in $A/(nA)$ one therefore has $x_n = 0$, as required.

Finally, we note that claim (iii) is both straightforward to prove directly and also follows immediately from the  general result \cite[Th. 15]{matlis} of Matlis (since $\widehat{A}^p$ is equal to the completion of the $\ZZ$-module $A$ at the ideal  generated by $p$). 
\end{proof} 

\subsubsection{}In the sequel we set $R_L^+ := \ZZ[\G_L^+]$ for each $L$ in $\Omega$ and consider the inverse limits 
\[ R^+ := \varprojlim_{L\in \Omega}R_L^+ \,\,\text{ and }\,\, \widehat{R^+} = \varprojlim_{L\in \Omega}\widehat{R_L^+}\]
where, in both cases, the transition morphisms are the natural projection maps. 

The following result is the main observation that we make in this section and will play a key role in the proof of Theorem \ref{main result}. 
 
\begin{proposition}\label{divisible prop} The diagonal map $X \to \prod_p\widehat{X}^p$, where $p$ runs over all primes, is injective. \end{proposition}

\begin{proof} The Chinese Remainder Theorem implies that the natural 
map $\widehat{X} \to \prod_p \widehat{X}^p$ is injective and so it is enough to prove that this is also true of the natural map $\iota: X \to \widehat{X}$.

Recalling that $C$ and $Y$ respectively denote the modules $Rc^{1+\tau} = R^+c^{1+\tau}$ and $Rv^{1+\tau} = R^+v^{1+\tau}$, we consider the following exact commutative diagram
\[
\begin{CD} 0 @> >> C @> \subseteq >> C + Y @> \pi >> X @> >> 0\\
& & @V\iota_2 VV @V\iota_1 VV @V \iota VV \\
0 @> >> \widehat{C} @> \subseteq >> (C + Y)^\wedge @> \widehat{\pi} >> \widehat{X} @> >> 0.\end{CD}\]
The top row of this diagram is the tautological short exact sequence, all vertical maps are the natural maps and the lower row is the short exact sequence that is induced by applying Lemma \ref{pro-completion lemma}(i) to the upper row and recalling that $X$ is torsion-free (by Lemma \ref{tf lemma}). 

The map $\iota_1$, and hence also $\iota_2$, is injective. This follows easily from the equality $\ker(\iota_1) = \bigcap_{n\in \mathbb{N}}n(C+Y)$  and the fact that for every $x$ in $C+Y$ and every $L$ in $\Omega$ the component $x_L$ of $x$ at $L$ belongs to the finitely generated group $U_L'$ defined in the proof of Proposition \ref{useful 3}(iii). 

We use $\iota_2$ and $\iota_1$ to regard $C$ and $C+Y$ as subgroups of $\widehat{C}$ and $(C+Y)^\wedge$ respectively, and then apply the Snake Lemma to the above diagram to deduce that the kernel of $\iota$ is isomorphic to the quotient of $\widehat{C} \cap (C+Y)$ by $C$, where the intersection takes place in $(C+Y)^\wedge$. 

To prove the claimed result we are therefore reduced to proving an equality  
\begin{equation}\label{sufficient} \widehat{C} \cap (C+Y) = C.\end{equation}

To check this we note that the map $R^+ \to C$ sending each element $r$ to $(c^{1+\tau})^r$ is bijective (as a consequence of \cite[Th. 1.2]{yonsei}) and hence extends to give an isomorphism $\widehat{R^+} \cong \widehat{C}$ of $\widehat{R^+}$-modules. 

It follows that every  element of $\widehat{C}$ is of the form $(c^{1+\tau})^\lambda$ with $\lambda = (\lambda_L)_L$ in $\widehat{R^+}$ and if such an element belongs to $C+Y$, and hence to ${\rm ES}(\mathbb{G}_m)^{1+\tau}$, then Proposition \ref{useful 3}(iii) implies that for each $L$ there exists a natural number $n_L$ such that $((c_L^{1+\tau})^{\lambda_L})^{n_L}$ belongs to the $R^+_L$-module $C_L$ that is generated by $c_L^{1+\tau}$. Thus, since $(c_L^{1+\tau})^{\lambda_L}$ belongs to $\widehat{C_L}$ and the quotient 
$\widehat{C_L}/C_L \cong C_L\otimes_\ZZ (\widehat{\ZZ}/\ZZ)$ is torsion-free, it follows that $(c_L^{1+\tau})^{\lambda_L}$ belongs to $C_L$. Since the annihilator of $c^{1+\tau}_L$ in $\widehat{R^+_L} = \widehat{\ZZ}\otimes_\ZZ R_L^+$ is equal to $\widehat{I_L}$ (as $\widehat{\ZZ}$ is a flat $\ZZ$-module), there must therefore exist an element $r_L$ of $R^+_L$ such that $\lambda_L-r_L \in \widehat{I_L}$.

It therefore remains to show that
\begin{equation}\label{key equal} \widehat{R^+} \cap \prod_{L\in \Omega} (R^+_L + \widehat{I_L}) = R^+.\end{equation}

To prove this equality we regard both $\widehat{\ZZ} = \prod_\ell\ZZ_\ell$ and $\QQ$ as subgroups of $\prod_\ell\QQ_\ell$ (where the products are over all primes $\ell$) in the natural way and note that, with these identifications, one has $\widehat{\ZZ} \cap \QQ = \ZZ$. 
To justify (\ref{key equal}) it is therefore enough to show that if $\lambda = (\lambda_L)_L$ is any element of $\widehat{R^+}$ with the property that $\lambda_L \in R^+_L + \widehat{I_L}$ for every $L$ in $\Omega$, then in fact one has $\lambda_L \in \QQ[\G_L^+]$ for every $L$. To prove this we shall argue by induction on the number of prime factors of the finite part $m(L)$ of the conductor of $L$.

If, firstly, $m(L)$ is a prime power, then the idempotent $e_L$ is equal to $1$ so Proposition \ref{useful 3}(i) implies 
$I_L$ vanishes and hence the given assumptions imply that $\lambda_L$ belongs to $R^+_L$, and hence also to $\QQ[\G_L^+]$ as required.

Now assume to be given a natural number $n$ and suppose that for every field $L$ in $\Omega$ such that $m(L)$ is divisible by at most $n$ primes, one has $\lambda_L \in \QQ[\G_L^+]$. Fix a field $F$ in $\Omega$ such that $m(F)$ is divisible by $n+1$ primes. We write $\Xi$ for the set of homomorphisms $\G_F^+ \to \QQ^{c\times}$ and for each $\psi$ in $\Xi$ we write $F_\psi$ for the fixed field of $F^+$ under $\ker(\psi)$. We note that, for each subfield $E$ of $F^+$ the subset $\Xi(E)$ of $\Xi$ comprising all $\psi$ for which $F_\psi = E$ is a (possibly empty) conjugacy class for the natural action of $\G_{\QQ^c}$ on $\Xi$ and hence that the associated idempotent $\varepsilon_E := \sum_{\psi \in \Xi(E)}e_\psi$ belongs to $\QQ[\G^+_F]$. 

To investigate $\lambda_F$ we use the decomposition 
\begin{multline}\label{lambda-decomposition}
    \lambda_F = 1\cdot \lambda_F = \left(\sum_{\psi\in \Xi} e_\psi\right)\cdot\lambda_F = \sum_{\psi\in \Xi} e_\psi \lambda_F \\ = \sum_{\psi\in \Xi} e_\psi \lambda_{F_\psi} = \sum_{E}\left(\sum_{\psi\in \Xi(E)}e_\psi\lambda_E\right) = \sum_E \varepsilon_E\lambda_E,\end{multline}
where the fourth equality is valid since $\lambda$ belongs to $\widehat{R^+}$, and in the sum $E$ runs over all subfields of $F^+$. 

Fix a subfield $E$ of $F^+$. If $m(E)$ is divisible by fewer primes than $m(F)$ then, by hypothesis, one has that $\lambda_E \in \QQ[\G_E]$. On the other hand, if $m(E)$ is divisible by the same number of primes as $m(F)$, and $r_F \in R^+_F$ and $i_F \in \widehat{I_F}$ are such that $\lambda_F = r_F + i_F$, then one has
\[ \varepsilon_E \lambda_E = \sum_{\psi \in \Xi(E)}e_\psi\lambda_{E} =\sum_{\psi \in \Xi(E)}e_\psi\lambda_{F} = \sum_{\psi \in \Xi(E)}e_\psi(r_F + i_F) = \sum_{\psi \in \Xi(E)}e_\psi r_F = \varepsilon_E r_F,
\]
where the fourth equality is valid since, under the present hypothesis, 
each $\psi$ in $\Xi(E)$ is not trivial on the decomposition group of any prime divisor of $m(F)$ so that one has  
 $e_\psi = e_\psi e_F$ (by Proposition \ref{useful 3}(ii)) and hence also $e_\psi (i_F) = 0$ as a consequence of 
 Proposition \ref{useful 3}(i). 
 
These observations imply that the element $\varepsilon_E\lambda_E$ belongs to $\QQ[\G_F^+]$ for every subfield $E$ of $F^+$ and hence, via the decomposition (\ref{lambda-decomposition}), that $\lambda_F$ belongs to $\QQ[\G_F^+]$, as required to complete the proof of the Proposition.
\end{proof}

\subsection{Euler systems of prime level}\label{prime section} The results of Lemma \ref{injectivity} and Proposition \ref{divisible prop} combine to imply that Theorem \ref{main result} is true provided the group $\widehat{X}^p$ vanishes for every prime $p$. In this section we reinterpret the vanishing of $\widehat{X}^p$ in terms of an explicit restriction on the components of the system $v^{1+\tau}$ at fields containing $\QQ(p)$. 

To do this we fix a prime $p$ and write $\Omega(p)$ for the subset of $\Omega$ comprising fields that contain $\QQ(p)$. We then define the collection ${\rm ES}_{(p)}(\mathbb{G}_m)$ of `Euler systems of level $p$' to be the set of elements $x = (x_L)_{L\in \Omega(p)}$ that are defined just as in Definition \ref{es def} except that all occurrences of $\Omega$ are replaced by $\Omega(p)$, and we write 
\[ \varrho: {\rm ES}(\mathbb{G}_m)\to {\rm ES}_{(p)}(\mathbb{G}_m)\]
 for the `restriction' map that sends each $(x_L)_{L\in \Omega}$ in ${\rm ES}(\mathbb{G}_m)$ to $(x_L)_{L \in \Omega(p)}$. 

We also set 
\[ \mathcal{E} = \mathcal{E}(p) := {\rm ES}_{(p)}(\mathbb{G}_m)^{1+\tau},\]
write $\lambda: \mathcal{E} \to \mathcal{E}^{\wedge,p}$ for the natural map and then for each system $x$ in ${\rm ES}(\mathbb{G}_m)$ we define  
\[ x_{(p)} := \lambda(\varrho (x^{1+\tau})) \in \mathcal{E}^{\wedge,p}.\]

Finally we set 
\[ R^+_p := \varprojlim_{L\in \Omega(p)}\ZZ_p[\G^+_L],\]
where the transition morphisms $L \subset L'$ in the limit are the natural projection maps. 

\begin{proposition}\label{last} Write $\kappa$ for the inclusion of $R c_{(p)}\cap Rv_{(p)}$ into $Rv_{(p)}$. Then the following claims are valid. 

\begin{itemize}
\item[(i)] The $R$-module $X$ is isomorphic to ${\rm cok}(\kappa)$. 
\item[(ii)] The natural map ${\rm cok}(\widehat{\kappa}^{p}) \to {\rm cok}(\kappa)^{\wedge, p}$ is bijective.
\item[(iii)] For every element $x$ of $Rc + Rv$, the natural map $R^+_p x_{(p)}\to (Rx_{(p)})^{\wedge,p}$ is bijective. 
\item[(iv)] Define subgroups of $\mathcal{E}^{\wedge,p}$ by setting $Z := Rc_{(p)} + Rv_{(p)}$ and $Z_p := R^+_pc_{(p)} + 
R^+_pv_{(p)}$. Then, with respect to the identifications in claim (iii), the image of $\widehat{\kappa}^{p}$ is equal to the set of elements $y$ of $R^+_pv_{(p)}$ which have the same image as an element of $R^+_pc_{(p)}$ under the 
natural map $Z_p \to \widehat{Z}^p$.  
\item[(v)] The group $\widehat{X}^p$ vanishes if there exists an element of $R^+_pc_{(p)}$ that has the same image as  $v_{(p)}$ under the map $Z_p \to \widehat{Z}^p$ that occurs in claim (iv). 
\end{itemize} \end{proposition} 

\begin{proof} To prove claim (i) we note that the association $x \mapsto \lambda(\varrho(x))$ induces a well-defined homomorphism of $R$-modules $t$ from $X = (Rc^{1+\tau} + Rv^{1+\tau})/Rc^{1+\tau}$ to the quotient $Q$ of $Rc_{(p)} + Rv_{(p)}$ by $Rc_{(p)}$. Since $Q$ is naturally isomorphic to ${\rm cok}(\kappa)$, it is thus enough to show that this map $t$ is bijective. Since $t$ is clearly surjective it is therefore enough to show that if $z$ is any element of $Rc^{1+\tau} + Rv^{1+\tau}$ such that $\lambda(\varrho(z))$ belongs to $Rc_{(p)}$, then $z$ belongs to $Rc^{1+\tau}$. 

To prove this we note first that $\lambda$ is injective. This is true since for every $x$ in $\mathcal{E}$ the component  $x_L$ at each field $L$ in $\Omega(p)$ belongs to the finitely generated torsion-free abelian group $(U_L')^{1+\tau}$. (We note in passing that this observation also implies that the group $\mathcal{E}$, and hence, by Lemma \ref{pro-completion lemma}(ii), also $\mathcal{E}^{\wedge,p}$, is torsion-free). 

The injectivity of $\lambda$ implies that if $\lambda(\varrho(z)) = c_{(p)}^r = \lambda(\varrho(c^{1+\tau}))^r$ for some $r$ in $R$, then the system $zc^{-(1+\tau)r}$ belongs to both ${\rm ES}(\mathbb{G}_m)^{1+\tau}$ and $\ker(\varrho)$. Thus, after converting between the notions of Euler system and circular distribution (as per the discussion in \S\ref{intro}), we can apply the result of \cite[Lem. 2.1]{yonsei} (in which we take $\Sigma$ to be the set of multiples of $p$, and we note that the notion of `circular distribution of level $p$' in loc. cit. corresponds to our notion of Euler system of level $p$) in order to deduce that $zc^{-(1+\tau)r} = 1$. This equality in turn implies that $z = c^{(1+\tau)r}$ belongs to $Rc^{1+\tau}$, as required to prove claim (i). 

To prove claim (ii) we note that the isomorphism in claim (i) combines with Lemma \ref{tf lemma} to imply ${\rm cok}(\kappa)$ is torsion-free. Given this, the tautological exact sequence 
\[ 0 \to R c_{(p)}\cap Rv_{(p)} \xrightarrow{\kappa} Rv_{(p)} \to {\rm cok}(\kappa) \to 0\]
combines with Lemma \ref{pro-completion lemma}(i) to imply that the induced sequence 
\[ 0 \to (R c_{(p)}\cap Rv_{(p)})^{\wedge,p} \xrightarrow{\widehat{\kappa}^p} (Rv_{(p)})^{\wedge,p} \to {\rm cok}(\kappa)^{\wedge,p} \to 0\]
is exact, and this immediately implies the isomorphism in claim (ii). 

To prove claim (iii) we fix $x \in Rc + Rv$ and for $L$ in $\Omega(p)$ write $J_L$ for the annihilator of $x_L^+$ in $\ZZ[\G_L^+]$. Then an element $r = (r_L)_{L\in \Omega(p)}$ of $R^+ = \varprojlim_{L\in \Omega(p)}\ZZ[\G_L^+]$ annihilates $x_{(p)}$ if and only if $r_L \in J_L$ for every $L\in \Omega(p)$ and so the annihilator of $x_{(p)}$ in $R^+$ is equal to the ideal $J := \varprojlim_{L\in \Omega(p)}J_L$. Since the $R$-module generated by $x_{(p)}$ is torsion-free, Lemma \ref{pro-completion lemma}(i) gives rise to an exact sequence of $R^+_p$-modules 
\begin{equation}\label{seq1} 0 \to \widehat{J}^p \xrightarrow{\subset} R^+_p \to (Rx_{(p)})^{\wedge,p} \to 0\end{equation}
in which the third arrow sends $1$ to the image of $x_{(p)}$ in $(Rx_{(p)})^{\wedge,p}$. 

Next we note that, since $\ZZ_p$ is flat over $\ZZ$, for each $L$ in $\Omega(p)$ the annihilator in $\ZZ_p[\G^+_L]$ of the $L$-component of $x_{(p)}$ is equal to $J_{L,p} := \ZZ_p\otimes_\ZZ J_L$. This implies that there is an exact sequence of $R^+_p$-modules
\begin{equation}\label{seq2} 0 \to J_p\xrightarrow{\subset} R^+_p \to R^+_px_{(p)} \to 0\end{equation}
where we set $J_p := \varprojlim_{L\in \Omega(p)}J_{L,p}$ and the third arrow sends $1$ to $x_{(p)}$. 

Now the groups $(Rx_{(p)})^{\wedge,p} $ and $R^+_px_{(p)}$ are torsion-free and for each $n$ the natural map
\[ ((Rx_{(p)})^{\wedge,p})/p^n \to Rx_{(p)}/p^n = (R^+_px_{(p)})/p^n\]
is bijective. Hence, if we take cokernels under multiplication by $p^n$ of the sequences (\ref{seq1}) and (\ref{seq2}) we obtain an equality 
\[ \widehat{J}^p/p^n = J_p/p^n.\]
In addition, from the exactness of each sequence 
\[ 0 \to J_{L,p}\xrightarrow{p^n} J_{L,p} \to J_{L,p}/p^n\to 0,\]
and the compactness of each module $J_{L,p}$, 
one finds that 
\[ J_p/p^n = \varprojlim_{L\in \Omega(p)}(J_{L,p}/p^n).\] 

Upon combining these observations, one deduces that 
\[\widehat{J}^p = \varprojlim_n\bigl(\widehat{J}^p/p^n\bigr) = \varprojlim_n\bigl(\varprojlim_{L\in \Omega(p)}(J_{L,p}/p^n)\bigr) = \varprojlim_{L\in \Omega(p)}\bigl(\varprojlim_n (J_{L,p}/p^n)\bigr) = \varprojlim_{L\in \Omega(p)} J_{L,p} = J_p,\]
where the first equality is valid by Lemma \ref{pro-completion lemma}(iii) and the fourth since $J_{L,p}= (J_L)^{\wedge,p}$ as $J_L$ is finitely generated. 
%
%
Then, since $\widehat{J}^p = J_p$, the assertion of claim (iii) follows directly upon comparing the exact sequences (\ref{seq1}) and (\ref{seq2}). 

To prove claim (iv) we consider the submodule $W := Rc_{(p)}\cap Rv_{(p)}$ of $\mathcal{E}$ and use the exact sequence of 
$R$-modules
\[ 0 \to W \xrightarrow{z\mapsto (z,z)} Rc_{(p)} \oplus Rv_{(p)} \xrightarrow{\theta} Z \to 0\]
in which $\theta$ sends each element $(x,y)$ to $x-y$. Now, since $\mathcal{E}$, and hence also $Z$, is torsion-free this sequence combines with Lemma \ref{pro-completion lemma}(i) and the isomorphisms in claim (iii) to imply exactness of the row in the following commutative diagram 
%
%
\[ \xymatrix{ 0 \ar[r] & \widehat{W}^p \ar[drr]_{\widehat{\kappa}^p} \ar[rr]^{\hskip -0.4truein x\mapsto (x,x)} & & R_p^+c_{(p)} \oplus R_p^+v_{(p)} \ar[d]^{(x,y) \mapsto y} \ar[r]^{\hskip 0.4truein\widehat{\theta}^{p}} & \widehat{Z}^{p} \ar[r] &0\\
& & & R^+_pv_{(p)}}\]
This exact diagram leads directly to the explicit description of $\im(\widehat{\kappa}^{p})$ given in claim (iv). 

Finally, to verify claim (v), we note that claim (i) implies $\widehat{X}^p$ vanishes if ${\rm cok}(\kappa)^{\wedge, p}$ vanishes and hence therefore, by claim (ii), if the map $\widehat\kappa^p$ is surjective. Claim (v) is therefore true since  claim (iv) implies that $\widehat\kappa^p$ is surjective if there exists an element of $R^+_pc_{(p)}$ that has the same image as  $v_{(p)}$ under the natural map $Z_p \to \widehat{Z}^p$. 
  \end{proof} 

\subsection{Completion of the proof} In view of Lemma \ref{injectivity}, Proposition \ref{divisible prop} and Proposition \ref{last}(v), to prove Theorem \ref{main result} it is enough to show that, for every prime $p$, the restricted system $v_{(p)}$ belongs to $R^+_pc_{(p)}$. We shall now explain how the latter claim follows as a consequence of results in  \cite{yonsei}. 

As a first step we note that, after converting between the notions of circular distribution and Euler system (just as in the proof of Proposition \ref{last}(i)), the result of \cite[Th. 3.1]{yonsei} implies that for each field $L$ in $\Omega(p)$ there exists an element $r_L$ of $\ZZ_p[\G_L]$ such that $v_L = r_L (c_L)$ in $\ZZ_p\otimes_\ZZ U_L'$ and hence also 
\[ v_L^{1+\tau} = r^+_L (c^{1+\tau}_L)\]
in $\ZZ_p\otimes_\ZZ (U'_L)^{1+\tau}$, where $r_L^+$ denotes the projection of $r_L$ to $\ZZ_p[\G_L^+]$. 

Now, since the system $\tilde v := (v_L^{1+\tau})_{L\in \Omega(p)}$ both belongs to the group 
$\mathcal{E} = {\rm ES}_{(p)}(\mathbb{G}_m)^{1+\tau}$ discussed in \S\ref{prime section}, and also verifies the above displayed equality (for a suitable choice of element $r_L^+$ of $\ZZ_p[\G_L^+]$) for every $L$ in $\Omega(p)$, it defines an element of the group 
$\mathcal{V}^{\rm d}_p$ discussed in \cite[\S5.3.1]{yonsei}. Hence, as the result of \cite[Prop. 5.3(i)]{yonsei} (in which the algebra $R_p^+$ is denoted by $\Lambda_{(p)}$) implies that $\mathcal{V}_p^{\rm d}$ is a free module over $R_p^+$, with basis given by the element $\tilde c := (c_L^{1+\tau})_{L\in \Omega(p)}$ of $\mathcal{E}$, there exists an element $r_p = (r_{p,L})_{L \in \Omega(p)}$ of $R_p^+$ with the property that, for every $L$ in $\Omega(p)$, one has 
\begin{equation}\label{last eq} v_L^{1+\tau} = r_{p,L} (c^{1+\tau}_L)\end{equation}
in $\ZZ_p\otimes_\ZZ (U'_L)^{1+\tau} = \varprojlim_n\bigl((U'_L)^{1+\tau}/p^n(U'_L)^{1+\tau}\bigr)$. 

To interpret these equalities, we recall that the systems $v_{(p)}$ and $c_{(p)}$ are respectively defined to be the images of $\tilde v$ and $\tilde c$ under the canonical map $\lambda: \mathcal{E}\to \mathcal{E}^{\wedge,p}$. In particular, if for each $n$ we fix an element $x_n$ of $R^+$ that has the same image under the projection map $R^+ \to \ZZ[\G_L^+]/p^n$ as does $r_p$ under the projection map $R_p^+ \to \ZZ_p[\G_L^+]/p^n = \ZZ[\G_L^+]/p^n$, and we write $\lambda_n$ for the canonical map $\mathcal{E} \to \mathcal{E}/p^n\mathcal{E}$, then $r_p(c_{(p)})$ is equal to the element $\bigl( \lambda_n(x_{n+1} (\tilde c))\bigr)_n$ of 
$\mathcal{E}^{\wedge,p} = \varprojlim_n\bigl(\mathcal{E}/p^n\mathcal{E}\bigr)$. 

In terms of this notation, the equalities (\ref{last eq}) imply that for every $L$ in $\Omega(p)$, the $L$-components $v_L^{1+\tau}$ and 
 $x_{n+1}(c_L^{1+\tau})$ of $\tilde v$ and $x_{n+1}(\tilde c)$ differ by an element of $p^{n+1}(U'_L)^{1+\tau}$ and thus, by Lemma \ref{last lemma} below, that the systems $\tilde v$ and $x_{n+1}(\tilde c)$ themselves differ by an element of $p^n\mathcal{E}$. 
 
It follows that $\lambda_n(\tilde v) = \lambda_n(x_{n+1}(\tilde c))$ for every $n$, and hence that there is an equality of systems
\[ v_{(p)} = \lambda(\tilde v) = (\lambda_n(\tilde v))_n = \bigl(\lambda_n(x_{n+1}(\tilde c))\bigr)_n = r_p(c_{(p)}).\]
Since this equality implies that $v_{(p)}$ belongs to $R_p^+c_{(p)}$, it therefore completes the proof of Theorem \ref{main result}. 

\begin{lemma}\label{last lemma} Fix a system $\varepsilon = (\varepsilon_L)_{L \in \Omega(p)}$ in $\mathcal{E}$ and a natural number $n$ such that $\varepsilon_{L}$ is divisible by $p^{n+1}$ in $(U'_{L})^{1+\tau}$ for every $L$. Then $\varepsilon$ is divisible by $p^n$ in $\mathcal{E}$.  
\end{lemma} 

\begin{proof} For each $E$ in $\Omega(p)$, the group $(U_E')^{1+\tau}$ is torsion-free and so the given hypotheses imply the existence of a unique element $y_E$ of $(U_E')^{1+\tau}$ with $\varepsilon_E = y_E^{p^{n+1}}$. For the same reason, the system $y := (y_E)_{E\in \Omega(p)}$ inherits from $\varepsilon$ the distribution relation (\ref{classical dr}) for all $E'$ and $E$ in $\Omega(p)$, and so belongs to ${\rm ES}_{(p)}(\mathbb{G}_m)$. The system $y^2 = y^{1+\tau}$ therefore belongs to $\mathcal{E}$. 

In particular, if $p=2$, then, since $\varepsilon_E = (y_E^{2})^{2^{n}}$ for every $E$ in $\Omega(p)$, the system $\varepsilon$ is equal to $(y^2)^{2^{n}}$ and so is divisible by $2^n$ in $\mathcal{E}$. 

Similarly, if $p$ is odd, then one has $\varepsilon^2_E = (y_E^{2})^{p^{n+1}}$ for every $E$ in $\Omega(p)$, so that the system $\varepsilon^2 = (y^{2p})^{p^{n}}$, and hence also the system $\varepsilon$ itself, is divisible by $p^n$ in $\mathcal{E}$. \end{proof}

\end{document}